\RequirePackage{amsmath}
\documentclass[runningheads]{llncs}
\usepackage{amssymb}
\usepackage{bbding}
\usepackage[T1]{fontenc}
\usepackage{graphicx}
\usepackage{hyperref}

\newcommand{\rvline}{\hspace*{-\arraycolsep}\vline\hspace*{-\arraycolsep}}

\begin{document}

\title{Multigrid Preconditioning for FD-DLM Method in Elliptic Interface Problems}
\titlerunning{Multigrid for FD-DLM}
\author{Najwa Alshehri\inst{1,2}
\orcidID{0000-0002-2953-5557} 
\and
Daniele Boffi\inst{1,3} 
% \Envelope
\orcidID{0000-0003-0717-2940} 
\and
Chayapol Chaoveeraprasit\inst{1,4}
\orcidID{0009-0009-4927-1649}
}
\authorrunning{N. Alshehri et al.}
\institute{
King Abdullah University of Science and Technology, K.S.A. \and
Jubail Industrial College, K.S.A.\and
Dipartimento di Matematica ``F. Casorati'', University of Pavia, Italy \and
National University of Singapore, Singapore\\
\email{najwa.alshehri@kaust.edu.sa}\\ \email{daniele.boffi@kaust.edu.sa}\\
\email{chayapol@u.nus.edu}
}

\maketitle             
\begin{abstract}
We investigate the performance of multigrid preconditioners for solving linear systems arising from finite element discretizations of elliptic interface problems using the Fictitious Domain with Distributed Lagrange Multipliers (FD-DLM) formulation. 
Numerical experiments are conducted using continuous and discontinuous finite element spaces for the Lagrange multiplier.
Results indicate that multigrid is a promising preconditioner for problems in the FD-DLM formulation.

\end{abstract}

\section{Introduction}
Elliptic interface problems are a significant area of study across various scientific and engineering disciplines due to their broad applicability. 
One prominent example is fluid-structure interaction (FSI), where a fluid and a solid structure exhibit distinct physical properties on either side of the interface.

A newly developed method in \cite{boffi_finite_2015}, called FD-DLM, employs finite element discretization with unfitted meshes, coupled with the fictitious domain method and distributed Lagrange multipliers. 
This method eliminates the need for re-meshing in time-dependent problems. 
The application of FD-DLM to elliptic interface problems using continuous elements for the multiplier was studied in \cite{auricchio_fictitious_2015,boffi_mixed_2014} and for discontinuous elements in \cite{alshehri_unfitted_2024}. 
In both cases, the method leads to a saddle point problem with a $3 \times 3$ block matrix structure, often resulting in ill-conditioned systems that require effective preconditioners.

The behavior of the condition number of the system in the case of FD-DLM applied to FSI problem was studied in~\cite{conditioning}. To our knowledge, only \cite{boffi_parallel_2024} has explored preconditioning strategies for FD-DLM, focusing on direct inversion techniques. 
Multigrid is an effective approach for accelerating iterative solvers by reducing both low- and high-frequency errors, making it highly efficient in solving large-scale problems \cite{briggs_multigrid_2000}. 
When applied to saddle point problems, multigrid can keep iteration counts nearly independent of the mesh size \cite{schilders_preconditioning_2008}.

This article introduces a multigrid preconditioning approach for elliptic interface problems in the FD-DLM formulation, considering both continuous and discontinuous Lagrange multipliers. 
Section 2 formulates the problem and the finite element discretization. 
Section 3 presents the preconditioners, while the multigrid strategies are introduced in Section 4.
Numerical results evaluating the performance of the preconditioners are provided in Section 5, followed by discussions on the obtained results.

\section{Problem Formulation}
We follow the formulation in \cite{alshehri_unfitted_2024}, restated here.

Let $\Omega$ be a domain in $\mathbb{R}^d$, where $d \in \{2,3\}$, with a bounded Lipschitz boundary $\partial \Omega$.
Let $\Omega_1$ and $\Omega_2$ be two subdomains of $\Omega$ such that $\overline{\Omega} = \overline{\Omega}_1 \cup \overline{\Omega}_2$, and let the interface $\Gamma = \overline{\Omega}_1 \cap \overline{\Omega}_2$ be Lipschitz continuous.
Assume that $\overline{\Gamma} \cap \partial \Omega = \emptyset$.

For a domain $\omega \subset \mathbb{R}^d$, let $L^2(\omega)$, $H^1(\omega)$ denote the standard Lebesgue and Sobolev spaces on $\omega$, respectively, endowed with their natural norms, $\|\phi\|_{L^2}$ and $\|\phi\|_{H^1}$ . 
The dual space of $H^1(\omega)$ is denoted by $[H^{1}(\omega)]^*$.
The dual norm of a functional $F \in [H^{1}(\omega)]^*$ is defined as $\|F\|_{[H^{1}(\omega)]^*} = \sup_{\phi \in H^1(\omega) \setminus \{0\}} \frac{|F(\phi)|}{\|\phi\|_{H^1}}.$

\begin{problem} \label{prob:orig}
Given $f_1: \Omega_1 \to \mathbb{R}, f_2: \Omega_2 \to \mathbb{R}$, $\beta_1 \in L^\infty(\Omega_1)$, $\beta_2 \in L^\infty(\Omega_2)$, such that $\inf \beta_1 > 0$ and $\inf \beta_2 > 0$,
find $(u_1,u_2): \Omega_1 \times \Omega_2  \to \mathbb{R}$ such that:
\begin{align}
-\nabla \cdot (\beta_i \nabla u_i) &= f_i &&\text{in } \Omega_i,\ i = 1,2, \nonumber\\
u_1 &= u_2 &&\text{on } \Gamma, \\
\beta_1 \nabla(u_1) \cdot \mathbf{n}_1 &= -\beta_2 \nabla(u_2) \cdot \mathbf{n}_2 &&\text{on } \Gamma, \nonumber\\
u_1 &= 0 &&\text{on } \partial \Omega_1,\nonumber
\end{align}
where $\mathbf{n}_i$ ($i = 1,2$) is the unit vector pointing out of $\Omega_i$ and normal to $\Gamma$.
\end{problem}

The FD-DLM method in \cite{auricchio_fictitious_2015} extends $\Omega_1$ to $\Omega$, and extends $u_1$, $f_1$, $\beta_1$ to $u, f$ and $\beta$ in $\Omega$, respectively, so that $u|_{\Omega_1} = u_1$, $f|_{\Omega_1} = f_1$ and $\beta|_{\Omega_1} = \beta_1$.
To ensure that the extended solution $u$ is equal to $u_2$ when restricted to $\Omega_2$, a Lagrange multiplier term is added. 
The weak variational formulation of Problem \ref{prob:orig} reads as follows.

\begin{problem} \label{prob:vari}
Given $f \in L^2(\Omega)$, $f_2 \in L^2(\Omega_2)$, $\beta \in L^\infty(\Omega)$ and $\beta_2 \in L^\infty(\Omega_2)$ with $f|_{\Omega_1} = f_1$ and $\beta|_{\Omega_1} = \beta_1$, find $(u,u_2,\lambda) \in V \times V_2 \times \Lambda$ such that:
\begin{align}
(\beta \nabla u, \nabla v)_\Omega + c(\lambda, v|_{\Omega_2}) &= (f,v)_\Omega &&\forall v \in V, \nonumber\\
((\beta_2 - \beta)\nabla u_2, \nabla v_2)_{\Omega_2} - c(\lambda, v_2) &= (f_2 - f, v_2)_{\Omega_2} &&\forall v_2 \in V_2, \\
c(\mu, u|_{\Omega_2} - u_2) &= 0 &&\forall \mu \in \Lambda, \nonumber
\end{align}
where $V = H_0^1(\Omega) = \{v \in H^1(\Omega) : v = 0 \text{ on } \partial\Omega\}$, $V_2 = H^1(\Omega_2)$, and $c: \Lambda \times V_2 \to \mathbb{R}$ depends on the choice of $\Lambda$.
\end{problem}

There are two natural choices for the space $\Lambda$ of Lagrange multipliers.
The first choice, $\Lambda_1 = [H^{1}(\Omega_2)]^*$, with $c(\lambda,v_2) = \langle \lambda, v_2 \rangle$, the duality pairing, allows for both continuous and discontinuous discretizations for $\Lambda$, while the second choice, $\Lambda_2 = H^1(\Omega_2)$, with $c(\lambda,v_2) = (\lambda, v_2)_{H^1}$, is applicable only with continuous discretizations, since the discontinuous finite elements do not lie in $H^1$.

Problem \ref{prob:vari} can then be written in operator form as

\begin{equation} \label{eq:opform}
\begin{pmatrix}
A_1 & 0 &\rvline& C_1^T \\
0 & A_2 &\rvline& -C_2^T \\
\hline
C_1 & -C_2&\rvline & 0
\end{pmatrix}
\begin{pmatrix}
u \\
u_2 \\
\hline
\lambda
\end{pmatrix}
=
\begin{pmatrix}
F_1 \\
F_2 \\
\hline
0
\end{pmatrix},
\end{equation}
where $A_1$ and $A_2$ are the operators associated with the bilinear forms $(\beta \nabla u, \nabla v)_\Omega$ and $((\beta_2 - \beta) \nabla u_2, \nabla v_2)_{\Omega_2}$, respectively, and $(C_1, -C_2)$ is the operator pair associated with $c(\mu, u|_{\Omega_2} - u_2)$. 
The well-posedness of this problem was established in \cite{auricchio_fictitious_2015}. 

Problem \ref{prob:vari} is discretized independently on $\Omega$ and $\Omega_2$ using the finite element method. 
Let $V_h \subset V$, $V_{2,h} \subset V_2$, and $\Lambda_h \subset \Lambda$ be finite-dimensional subspaces spanned by piecewise polynomial basis functions.
The discrete problems can then be stated as follows:

\begin{problem} \label{prob:discrete} 
Given $f\in L^2(\Omega)$ and $f_2\in L^2(\Omega_2)$, find $(u_h,u_{2,h},\lambda_h)\in V_h\times V_{2,h}\times\Lambda_h$ such that
\begin{align}
(\beta\nabla u_h,\nabla v_h)_\Omega+c(\lambda_h,v_h|_{\Omega_2})&=(f,v_h)_\Omega &&\forall\nu_h\in V_h, \nonumber\\
((\beta_2-\beta)\nabla u_{2,h},\nabla\nu_{2,h})_{\Omega_2}-c(\lambda_h,\nu_{2,h})&=(f_2-f,\nu_{2,h})_{\Omega_2} &&\forall\nu_{2,h}\in V_{2,h}, \\
c(\mu_h,u_h|_{\Omega_2}-u_{2,h})&=0 &&\forall\mu_h\in\Lambda_h,  \nonumber  
\end{align}
where $c(\cdot,\cdot)$ is the $L^2$ scalar product when $\Lambda =\Lambda_1$ and the $H^1$ scalar product when $\Lambda = \Lambda_2$.
\end{problem}

We consider quadrilateral meshes. Let $Q_1$ denote the space of continuous piecewise bilinear polynomials, and let $B$ denote the space of bubble functions that vanish on the boundary of each element.
The bubble functions are biquadratic polynomials in the case of a quadrilateral mesh in a two-dimensional domain.
Let $P_0$ denote the space of piecewise constant polynomials. 
In this work, we consider two choices of finite elements discretizations.

Element 1 $[Q_1 - Q_1 - Q_1]$:
The continuous piecewise linear space is used to discretize all of $V$, $V_2$, and $\Lambda$.
The stability of the $P_1 - P_1 - P_1$ element was studied in \cite{auricchio_fictitious_2015} and was extended by \cite{boffi_mixed_2014} to the case of quadrilateral meshes.

Element 2 $[Q_1 - (Q_1+B) - P_0]$:
The elements $Q_1$, $Q_1+B$, and $P_0$ are used to discretize $V$, $V_2$, and $\Lambda$, respectively.
Note the addition of the bubble function for $V_2$.
The a priori error analysis of this element was first studied in \cite{alshehri_unfitted_2024} where it was also observed that the bubble is necessary for the stability of the scheme, while the a posteriori analysis was conducted in \cite{alshehri_posteriori_2024}.
Since the discontinuous element $P_0$ is used, this element will be applied only to the case of $\Lambda_1 = [H^1(\Omega_2)]^*$.

In its matrix form, Problem \ref{prob:discrete} becomes a linear system of the same form as \eqref{eq:opform}, which we will denote as $A\vec{x} = \vec{b}$. 
It was proved in \cite{alshehri_unfitted_2024,auricchio_fictitious_2015,boffi_mixed_2014} that the discrete problem remains well-posed under the condition $\beta_2 > \beta$, provided that the finite element spaces $V$, $V_2$, and $\Lambda$ are discretized using either of the above two elements.
Some numerical results suggest that the condition $\beta_2 > \beta$ could be relaxed.

\section{Preconditioners}
Solving the linear system $A\vec{x} = \vec{b}$ can be computationally expensive, especially as the mesh is refined and the problem size grows.
Preconditioning is a widely used technique to accelerate convergence and improve the numerical stability of iterative solvers.
The original system is transformed into an equivalent form, $P^{-1} A \vec{x} = P^{-1} \vec{b}$, where $P$ is the preconditioner.
The preconditioner $P$ should approximate the original matrix $A$ while being easier to invert or apply, and should reduce the condition number of the system significantly, leading to faster convergence.\\
Let
\[
B = 
\begin{pmatrix}
A_2 &\rvline& -C_2^T \\
\hline
-C_2 &\rvline& 0
\end{pmatrix}.
\]
The three preconditioners under consideration are defined as follows:
\begin{align*}
P_1 &= 
\begin{pmatrix}
A_1 &\rvline& C_1^T \\
\hline
0 &\rvline& B
\end{pmatrix}, &&
P_2 = 
\begin{pmatrix}
A_1 &\rvline& 0 \\
\hline
C_1 &\rvline& B
\end{pmatrix}, &&
P_3 = 
\begin{pmatrix}
A_1 &\rvline& 0 \\
\hline
0 &\rvline& B
\end{pmatrix}.
\end{align*}

The inverse of these preconditioners is applied to reduce the condition number of the entire system and speed up the solving process.
To ensure the invertibility of the proposed preconditioners, and given that the system has a triangular or diagonal matrix structure, we must analyze the invertibility of the diagonal blocks, $A_1$ and $B$. 

The block $A_1$ mimics a standard elliptic problem with Dirichlet boundary conditions and is therefore invertible. The block $B$, on the other hand, does not involve any boundary conditions applied to the immersed domain.
To establish its invertibility, we note that $B$ is in a saddle-point format.
According to the theory of saddle-point problems \cite{boffi_mixed_2013}
the invertibility of $B$ requires verifying two conditions: 
1. $A_2$ is elliptic in the kernel of $C_2$.
2. $C_2$ is full row rank.
Note that $C_2$ vary depending on the choice of $\Lambda$.

From now on, let $C$ denote a generic positive constant that may change from one step to another.
\begin{proposition}\label{prop:elk2}
For $\beta_2 > \beta > 0$, there exists a constant $C > 0$ such that, for all $u_{2,h}$ in the kernel of $C_2$, the following inequality holds:
\[
((\beta_2-\beta)\nabla u_{2,h},\nabla u_{2,h})_{\Omega_2} \geq C \|u_{2,h}\|^2_{1,\Omega_2}, \quad \forall u_{2,h} \in V_{2,h}.
\]
\end{proposition}
\begin{proof}
We observe that:
\[
((\beta_2-\beta)\nabla u_{2,h},\nabla u_{2,h})_{\Omega_2} \geq C\; |u_{2,h}|^2_{1,\Omega_2}.
\]
In Element 1, where $\lambda$ is a piecewise bilinear polynomial,  $u_{2,h}$ in the kernel of $C_2$ means that the kernel contains only the trivial solution since $u_{2,h}$ and $\mu_h$ belong to the same space. On the other hand, in Element 2, where $\lambda$ is a piecewise constant spanned by $\{1\}$, $u_{2,h}$ is in the kernel of $C_2$ means that it has a zero mean value by the definition of elements in that kernel. 
Using the Poincaré inequality in both scenarios, we deduce:
\[
((\beta_2-\beta)\nabla u_{2,h},\nabla u_{2,h})_{\Omega_2} \geq C\; |u_{2,h}|^2_{1,\Omega_2} \geq C \;\|u_{2,h}\|^2_{1,\Omega_2}.
\]
This ensures that $A_2$ is elliptic in the kernel of $C_2$.
\qed
\end{proof}

\begin{proposition} \label{prop:inf-sup}
There exists a constant $C > 0$ such that, for all $(u_h, u_{2,h})$ in the kernel of $C_2$, the following inequality holds:
\[
\sup_{ u_{2,h} \in V_{2,h}} \frac{\left\langle \mu, -u_{2,h} \right\rangle}{ \|u_{2,h}\|_{1,\Omega_2}} \geq C \|\mu\|_{\Lambda},
\]
\end{proposition}
where the norm $\|\mu\|_{\Lambda}$ is given by either $\|\mu\|_{1,\Omega_2}$ or $ \|\mu\|_{[H^1(\Omega_2)]^*}$ depending on the choice of $\Lambda$.

This inf-sup condition was proven for Element 1 in \cite{auricchio_fictitious_2015} and for Element 2 in \cite{alshehri_unfitted_2024}.
Putting all results together, we obtain that $A_1$ and $B$ are invertible for both choices of finite element spaces and the inverse of the proposed preconditioners $P_1, P_2,$ and $P_3$ 	is given by:
\begin{align*}
P_1^{-1} &= 
\begin{pmatrix}
A_1^{-1}&\rvline & -A_1C_1^TB^{-1} \\
\hline
0&\rvline & B^{-1}
\end{pmatrix}, &&
P_2^{-1} = 
\begin{pmatrix}
A_1^{-1}&\rvline & 0 \\
\hline
-B^{-1}C_1A_1^{-1}&\rvline & B^{-1}
\end{pmatrix}, &&
P_3^{-1} = 
\begin{pmatrix}
A_1^{-1}&\rvline & 0 \\
\hline
0&\rvline & B^{-1}
\end{pmatrix}.
\end{align*}

\section{Multigrid Approximation}
In \cite{boffi_parallel_2024}, $A_1^{-1}$ and $B^{-1}$ 	were computed using direct solvers.
In contrast, we employ a multigrid approximation of $A_1^{-1}$ and $B^{-1}$.
Specifically, let $\widetilde{A_1^{-1}}$ denote a matrix such that for any vector $\vec{b}$, $\widetilde{A_1^{-1}}\vec{b}$ is equal to the solution $\vec{x}$ of $A_1\vec{x} = \vec{b}$ obtained by using a single multigrid $V$-cycle pass.
We used Symmetric Over-Relaxation (SOR) as the pre- and post-smoothers for the multigrid cycle applied to $A_1$.

It is natural to also apply multigrid to obtain $\widetilde{B^{-1}}$.
However, since the matrix $B$ has zeros on the diagonal, the usual smoothers such as SOR or Jacobi could not work, as they all require inversion of the diagonal.

After extensive research, we identified Vanka preconditioners \cite{vanka_block-implicit_1986} as a suitable alternative.
Vanka preconditioners were specifically designed for preconditioning saddle point problems or problems involving Lagrange multipliers, where the matrix often has a zero block on the diagonal.
The Vanka preconditioning process operates as follows: for each degree of freedom $i$ associated with the Lagrange multiplier, a set $D_i$ is constructed.
This set includes $i$ itself and all degrees of freedom $j$ in $\Omega_2$ such that the coupling entry between $i$ and $j$ is nonzero.
The entries corresponding to the coupling between any two degrees of freedom in $D_i$ are then extracted from $B$ to form a dense local matrix $B_i$.
The inverse $B_i^{-1}$ 	is then computed directly and applied to update the current solution vector $\vec{x}$.
A single pass of Vanka preconditioning 	involves executing this procedure for each $i$ associated with the Lagrange multiplier.

In the next section we conduct numerical experiments to evaluate the performance of multigrid preconditioners and to compare their effectiveness with preconditioners utilizing direct inversion techniques.
We implement preconditioners and solvers using deal.II \cite{africa_dealii_2024}.
The library has the built-in class \texttt{SparseVanka} for a Vanka smoother.
To enable its use as a multigrid preconditioner, we implemented the additional methods \texttt{Tvmult()} and \texttt{clear()} in this class.

For each preconditioner $P_i$, four different approaches are used, depending on whether a direct or multigrid preconditioner was used for $A_1^{-1}$ and $B^{-1}$.
For example, the label ``dm'' indicates the use of a direct solver for $A_1^{-1}$ and a multigrid preconditioner for $B^{-1}$.

\section{Numerical Results}
We concuct a series of numerical experiments under controlled conditions to assess the effectiveness of the proposed preconditioners.
The parameters and solver settings are carefully chosen to ensure a fair comparison.
We considered a two-dimensional problem with background domain $\Omega = [-1.4, 1.4]^2$ and immersed domain $\Omega_2 = B(0,1)$.
The coefficients are set to $ \beta = 1 $ and $ \beta_2 = 10 $ and the source terms to $ f =  f_2 = 1 $.

Each multigrid V-cycle uses two pre- and post-smoothing passes.  
However, if convergence is not achieved within the first cycle --- specifically for the case ``dm'' and ``mm'' under Element 2 discretization across all preconditioners $ P_1, P_2, $ and $ P_3 $ --- five smoothing passes are applied instead.
The solver terminates when the residual satisfies $ r < 10^{-12} $ or the iteration count exceeds $ 10^5 $.  
All computations are performed on a personal laptop without parallelization.

Figures \ref{fig:cn_h} and \ref{fig:it_t} each consist of nine subfigures, displaying the condition number versus mesh size, $ h $, and the iteration count versus solving time, respectively, on a log-log scale for the nine cases under examination.  
The rows correspond to different preconditioners, while the columns represent the finite element discretization cases.  

% Figure 1
\begin{figure}[ht] 
    \centering
    \begin{minipage}{0.32\textwidth}\centering
        \small \quad \quad Element:1 , $\Lambda =\Lambda_2$
    \end{minipage}
    \begin{minipage}{0.32\textwidth}\centering
        \small \quad \quad Element:1 , $\Lambda =\Lambda_1$
    \end{minipage}
    \begin{minipage}{0.32\textwidth}\centering
        \small \quad \quad Element:2 , $\Lambda =\Lambda_1$
    \end{minipage}

    \begin{minipage}{0.02\textwidth}\centering
        \small $P_1$
    \end{minipage}    
    \begin{minipage}{0.30\textwidth}
        \centering
        \includegraphics[width=\textwidth]{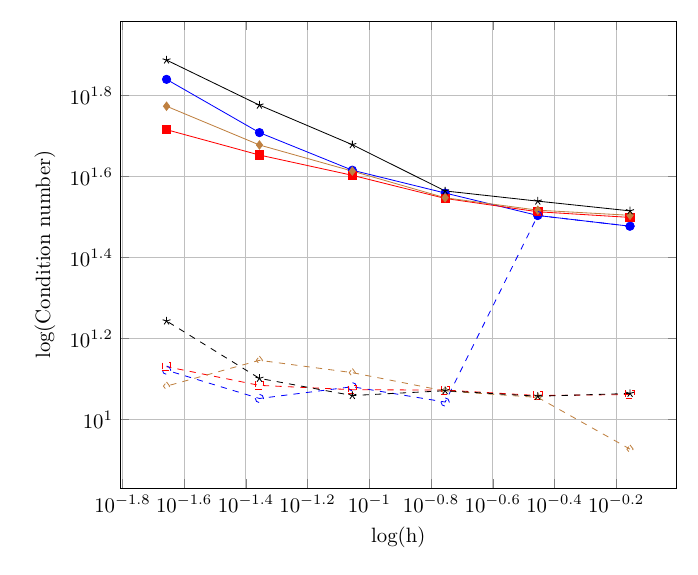}
    \end{minipage}
    \hfill
    \begin{minipage}{0.30\textwidth}
        \centering
        \includegraphics[width=\textwidth]{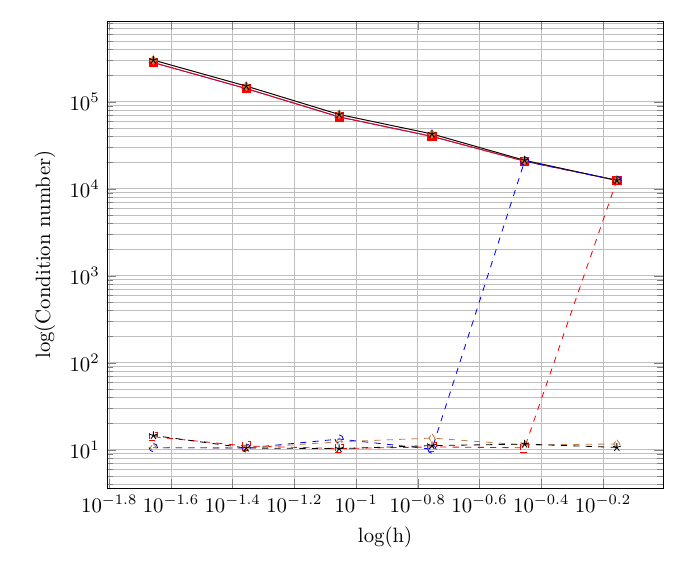}
    \end{minipage}
    \hfill
    \begin{minipage}{0.30\textwidth}
        \centering
        \includegraphics[width=\textwidth]{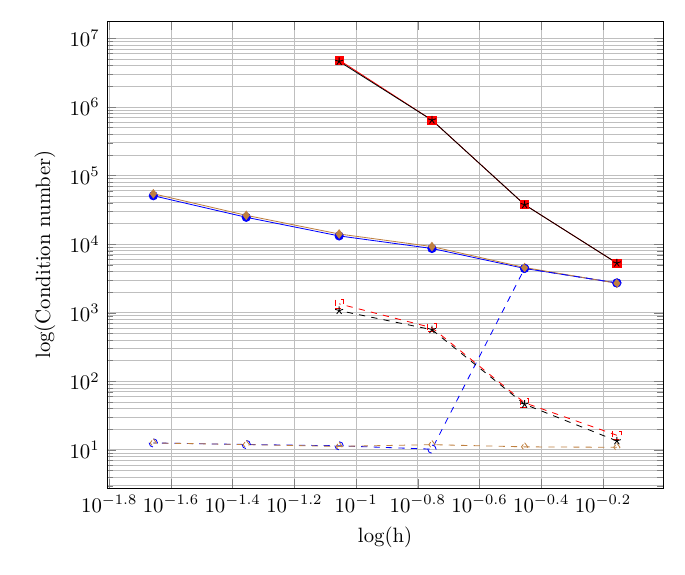}
    \end{minipage}

    \begin{minipage}{0.02\textwidth}\centering
        \small $P_2$
    \end{minipage}
    \begin{minipage}{0.30\textwidth}
        \centering
        \includegraphics[width=\textwidth]{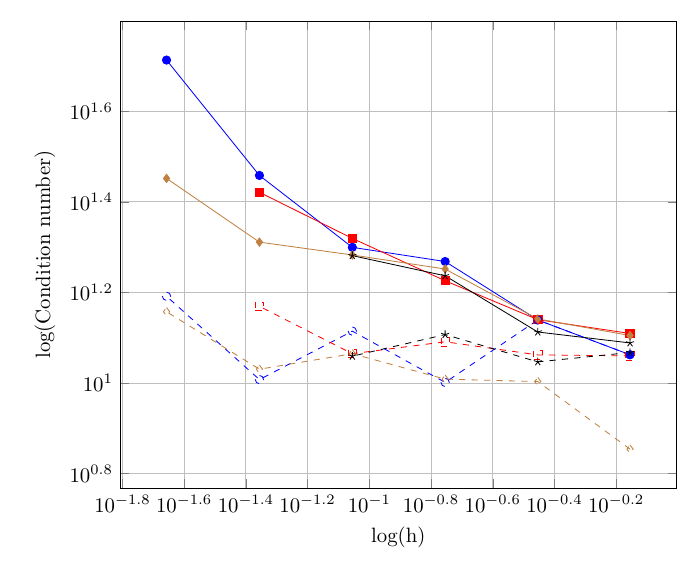}
    \end{minipage}
    \hfill
    \begin{minipage}{0.30\textwidth}
        \centering
        \includegraphics[width=\textwidth]{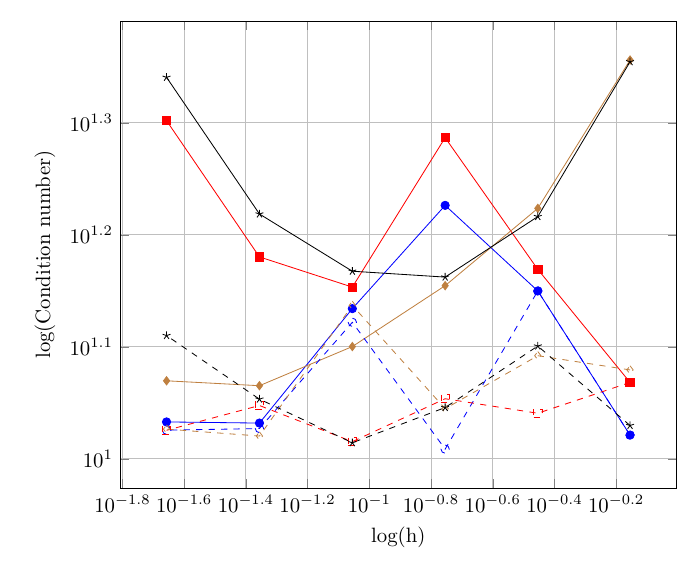}
    \end{minipage}
    \hfill
    \begin{minipage}{0.30\textwidth}
        \centering
        \includegraphics[width=\textwidth]{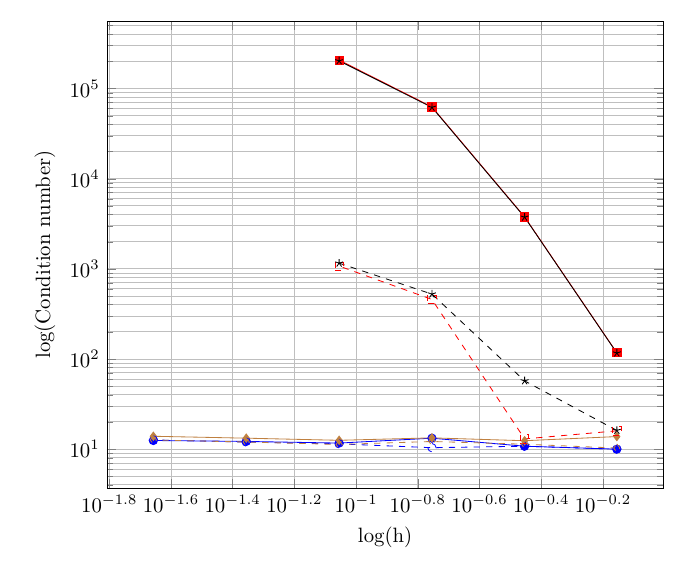}
    \end{minipage}
    
    \begin{minipage}{0.02\textwidth}\centering
        \small $P_3$
    \end{minipage}
    \begin{minipage}{0.30\textwidth}
        \centering
        \includegraphics[width=\textwidth]{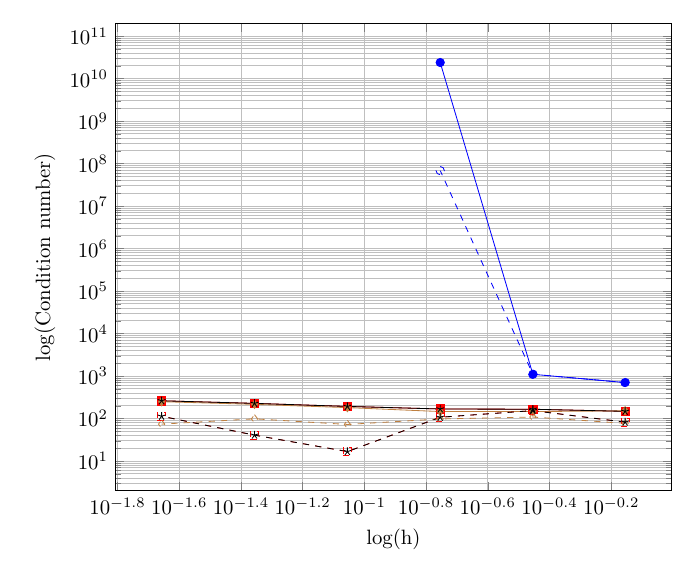}
    \end{minipage}
    \hfill
    \begin{minipage}{0.30\textwidth}
        \centering
        \includegraphics[width=\textwidth]{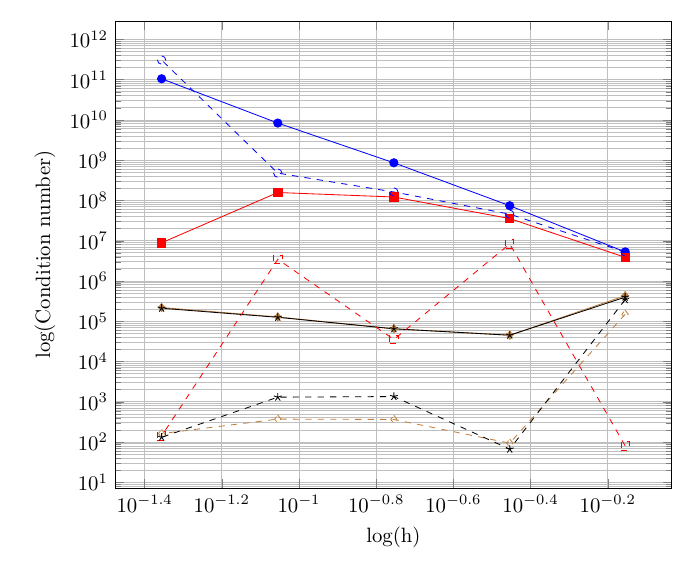}
    \end{minipage}
    \hfill
    \begin{minipage}{0.30\textwidth}
        \centering
        \includegraphics[width=\textwidth]{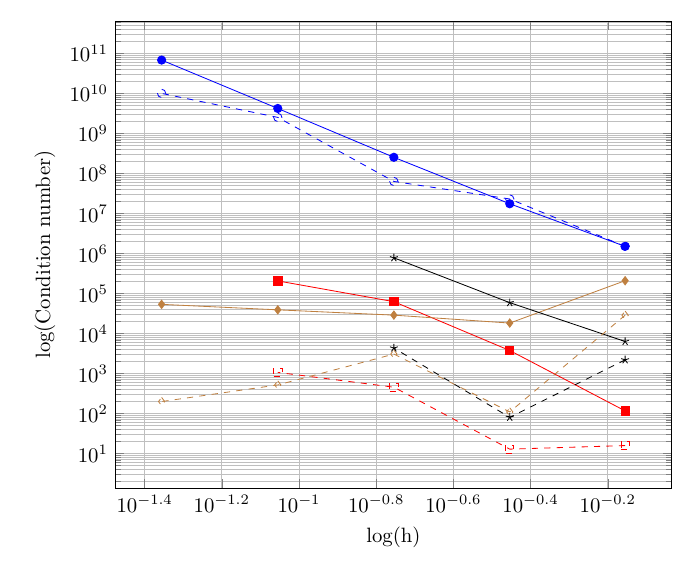}
    \end{minipage}
    
    \begin{minipage}{0.42\textwidth}\centering
        \includegraphics[width=\textwidth]{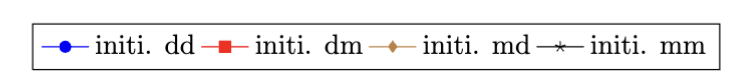}
    \end{minipage}
    \begin{minipage}{0.42\textwidth}\centering
        \includegraphics[width=\textwidth]{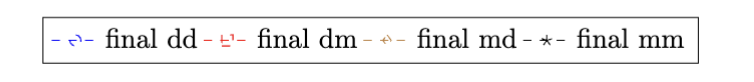}
    \end{minipage}
    \caption{Initial and final condition number of the full system matrix $ A $}
    \label{fig:cn_h}
\end{figure}

% Figure 2
\begin{figure}[ht]
    \centering
    \begin{minipage}{0.32\textwidth}\centering
        \small \quad \quad Element:1 , $\Lambda =\Lambda_2$
    \end{minipage}
        \begin{minipage}{0.32\textwidth}\centering
        \small \quad \quad Element:1 , $\Lambda =\Lambda_1$
    \end{minipage}
        \begin{minipage}{0.32\textwidth}\centering
        \small \quad \quad Element:2 , $\Lambda =\Lambda_1$
    \end{minipage}

    \begin{minipage}{0.02\textwidth}\centering
        \small $P_1$
    \end{minipage}    
    \begin{minipage}{0.30\textwidth}
        \centering
        \includegraphics[width=\textwidth]{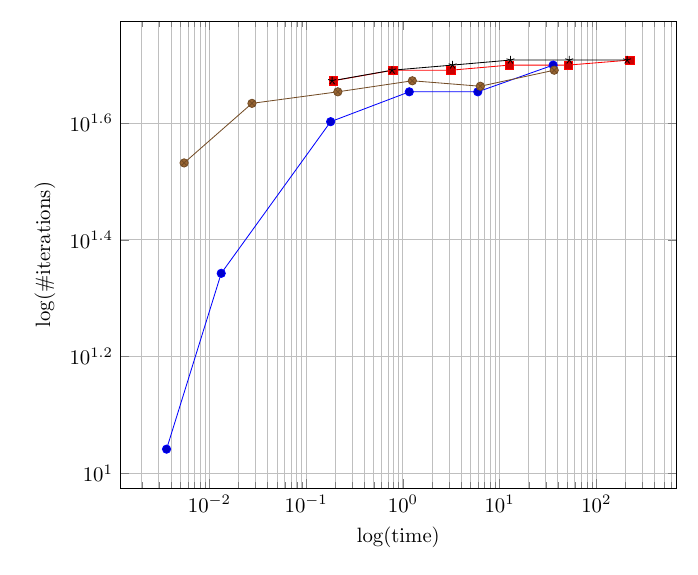}
    \end{minipage}
    \hfill
    \begin{minipage}{0.30\textwidth}
        \centering
        \includegraphics[width=\textwidth]{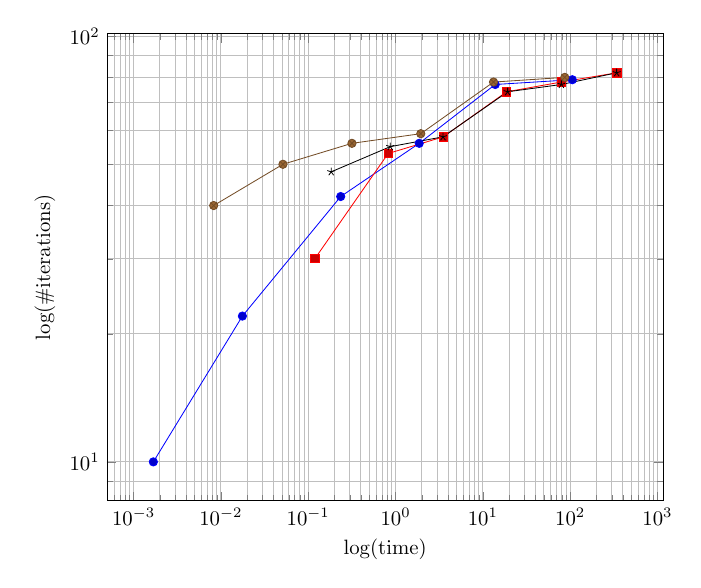}
    \end{minipage}
    \hfill
    \begin{minipage}{0.30\textwidth}
        \centering
        \includegraphics[width=\textwidth]{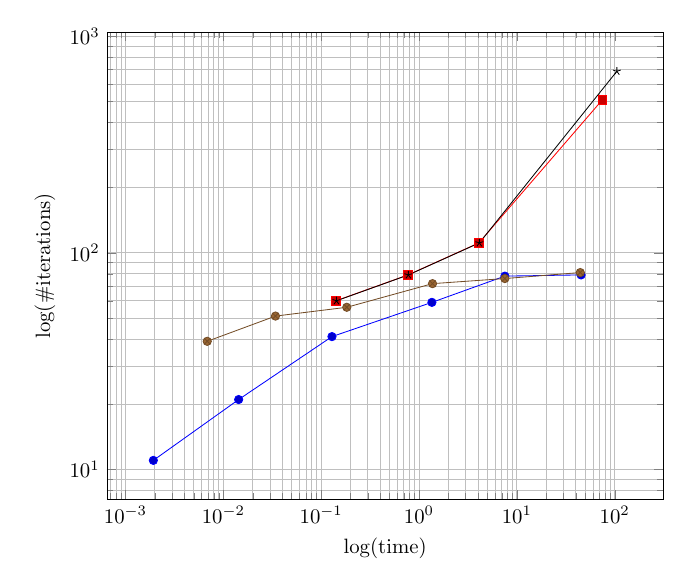}
    \end{minipage}
    
    \begin{minipage}{0.02\textwidth}\centering
        \small $P_2$
    \end{minipage}
    \begin{minipage}{0.30\textwidth}
        \centering
        \includegraphics[width=\textwidth]{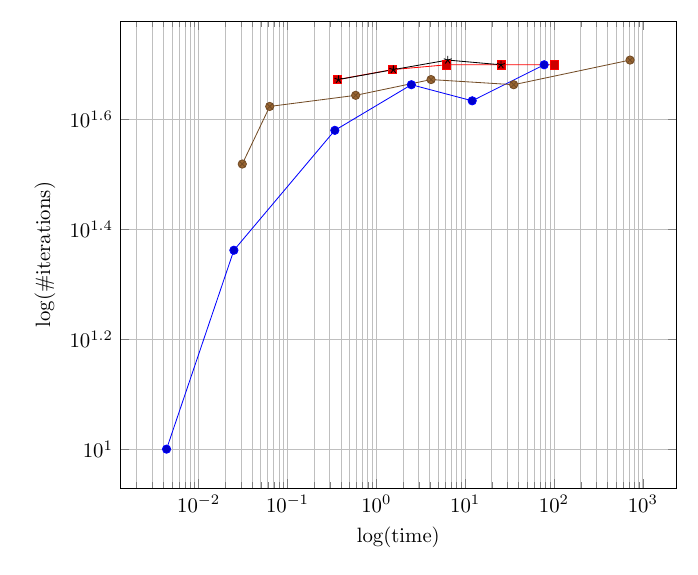}
    \end{minipage}
    \hfill
    \begin{minipage}{0.30\textwidth}
        \centering
        \includegraphics[width=\textwidth]{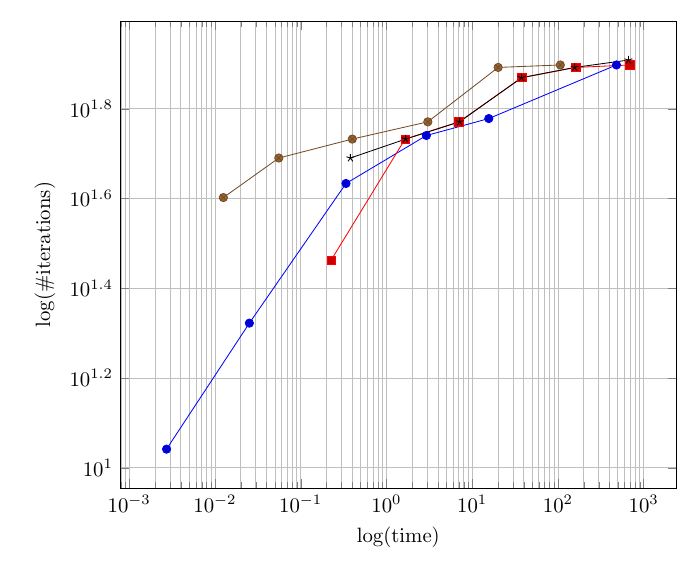}
    \end{minipage}
    \hfill
    \begin{minipage}{0.30\textwidth}
        \centering
        \includegraphics[width=\textwidth]{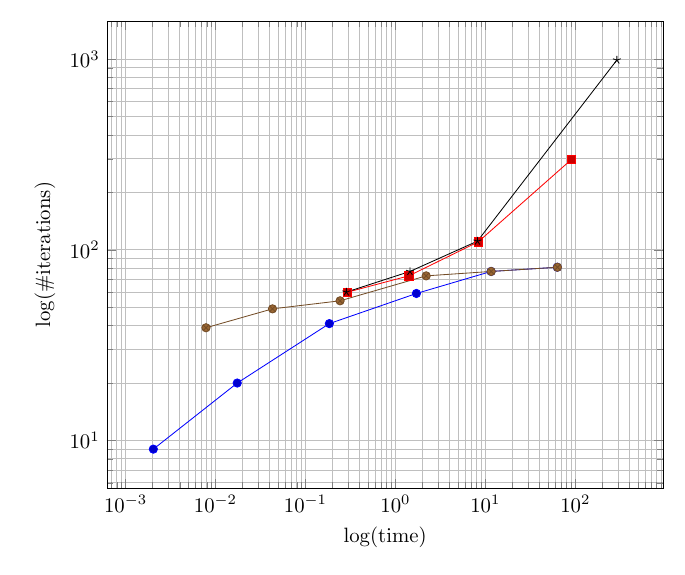}
    \end{minipage}
    
    \begin{minipage}{0.02\textwidth}\centering
        \small $P_3$
    \end{minipage}
    \begin{minipage}{0.30\textwidth}
        \centering
        \includegraphics[width=\textwidth]{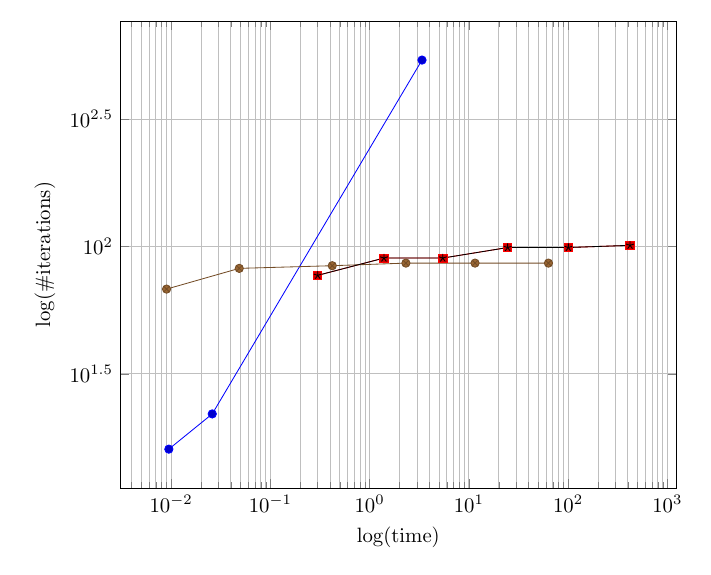}
    \end{minipage}
    \hfill
    \begin{minipage}{0.30\textwidth}
        \centering
        \includegraphics[width=\textwidth]{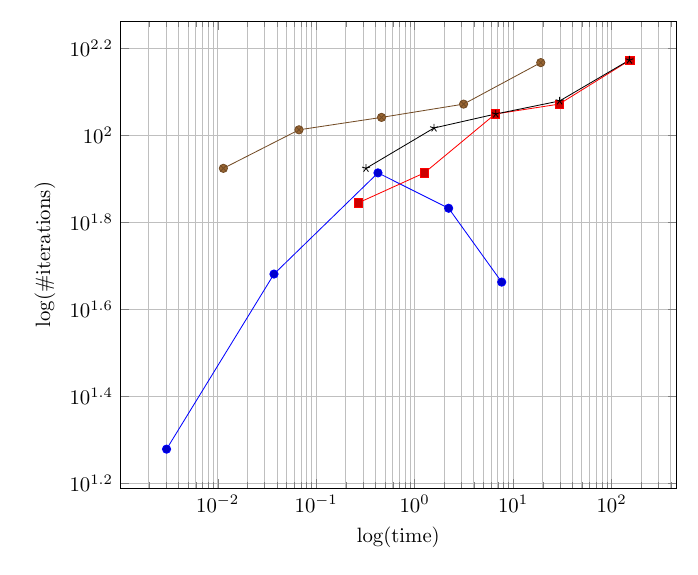}
    \end{minipage}
    \hfill
    \begin{minipage}{0.30\textwidth}
        \centering
        \includegraphics[width=\textwidth]{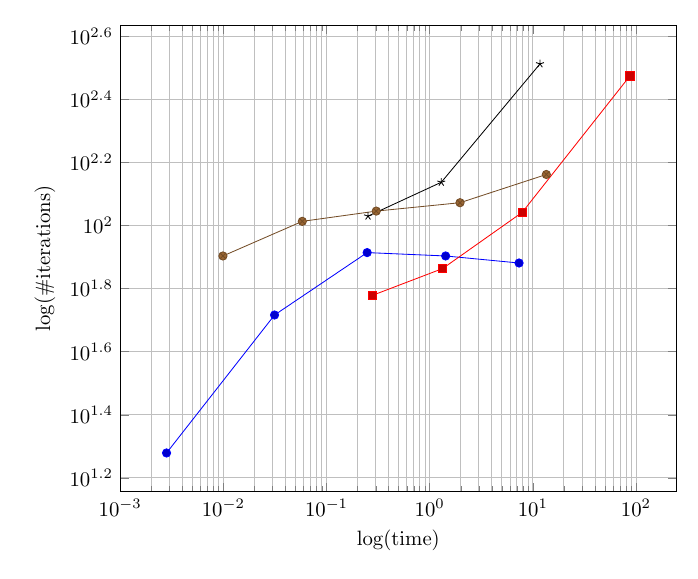}
    \end{minipage}

    \begin{minipage}{0.30\textwidth}\centering
        \includegraphics[width=\textwidth]{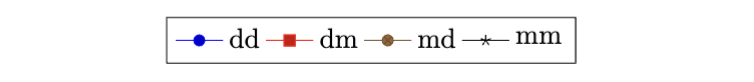}
    \end{minipage}
    \caption{Iteration number relative to the solving time}
    \label{fig:it_t}
\end{figure} 

A key observation is that the choice of preconditioner significantly influences computational efficiency.  
While some preconditioners exhibit robust performance across all discretizations, others show sensitivity to specific settings. 
In the following, we analyze these results.  

\textbf{Element 1 with $ \Lambda = \Lambda_2 $}  

For this case, where the Lagrange multiplier term is evaluated using the $ H^1 $ scalar product, preconditioners $ P_1 $ and $ P_2 $ 	demonstrate significant improvement in reducing the final condition number.  
Moreover, the final condition number remains stable and does not grow with mesh refinement.  
Our findings align with those reported in \cite{boffi_parallel_2024}, where the diagonal preconditioner $ P_3 $ with ``dd'' was found to be less effective in reducing the matrix condition number. However, our application of multigrid techniques—whether across the entire system or selectively to specific blocks (namely ``dm'', ``md'', and ``mm'')—substantially enhances the condition number, demonstrating considerable improvement.  
Additionally, we observe that in these cases, the iteration count versus solving time reaches a constant level under mesh refinement, which is a highly desirable outcome and an important improvement to existing results.  

\textbf{Element 1 with $ \Lambda = \Lambda_1 $}.
For this case, where the Lagrange multiplier term is evaluated using the $ L^2 $ scalar product, similar results are observed.  
It is evident that the condition number of the matrix in the initial case grows significantly faster compared to the case of $ \Lambda_2 $.  
However, all preconditioners --- particularly $ P_1 $ --- achieve a substantial reduction in the condition number.  
Preconditioner $ P_2 $ is also effective, although some oscillations are observed.  
In this instance, while the diagonal preconditioner $P_3$ with ``dd'' does not significantly reduce the condition number, this challenge is effectively addressed when using other preconditioners that incorporate multigrid techniques. 
These alternatives significantly enhance the condition number, leading to much more favorable outcomes. 
On the other hand, the iteration count increases slowly for all $P_3$ with multigrid involvement.

\textbf{Element 2 with $ \Lambda = \Lambda_1 $}.
For this case, where the Lagrange multiplier term is evaluated using the $ L^2 $ scalar product, it is evident that the Vanka preconditioner is not effective, highlighting the need for further improvements in the multigrid approach applied to the block $ B $.  

Nevertheless, the preconditioners ``dd'' and ``md'' in $P_1$ and $P_2$ successfully improve the condition number, making it independent of the mesh size.  
The diagonal preconditioner also shows improvement in the case of ``md''.  
Although the condition number does not decrease in the case of ``dd'', the iteration count stabilizes relative to the solving time, indicating consistent computational performance.  

These observations suggest that applying multigrid to $A$ and/or to $B$ is a promising preconditioning strategy.  
Further improvements could be achieved by optimizing the multigrid parameters and the number of smoothing cycles.

Comparing across different preconditioners, in all cases, the block upper triangular preconditioner $P_1$ and the block lower triangular preconditioner $P_2$ perform similarly and outperform the block diagonal preconditioner $P_3$. 
Moreover, ``md'' seems to work effectively for all test cases and preconditioners tested.

To further improve the proposed preconditioners, a theoretical investigation into convergence rates and condition number reduction would be beneficial.  
More advanced preconditioners could be developed based on the full $3 \times 3$ block structure of the system matrix $A$ , rather than treating it as a $2 \times 2$ block.  
It is important to note that any attempt to construct a preconditioner based on the inverse of the $A_2$ block is unlikely to succeed, as $A_2$ is singular due to the lack of boundary conditions on $\Omega_2$.  

\noindent
\textbf{Acknowledgment}: DB is member of INdAM Research group GNCS.

\bibliographystyle{splncs04}
\bibliography{ref}

\begin{thebibliography}{10}
\providecommand{\url}[1]{\texttt{#1}}
\providecommand{\urlprefix}{URL }
\providecommand{\doi}[1]{https://doi.org/#1}

\bibitem{africa_dealii_2024}
Africa, P.C., Arndt, D., Bangerth, W., Blais, B., Fehling, M., Gassmöller, R.,
  Heister, T., Heltai, L., Kinnewig, S., Kronbichler, M., Maier, M., Munch, P.,
  Schreter-Fleischhacker, M., Thiele, J.P., Turcksin, B., Wells, D., Yushutin,
  V.: The deal.{II} library, {Version} 9.6. Journal of Numerical Mathematics
  \textbf{32},  369--380 (2024). \doi{10.1515/jnma-2024-0137}

\bibitem{alshehri_posteriori_2024}
Alshehri, N., Boffi, D., Gastaldi, L.: A posteriori error estimator for
  elliptic interface problems in the fictitious formulation. Available at:
  \url{http://arxiv.org/abs/2407.00786} (2024)

\bibitem{alshehri_unfitted_2024}
Alshehri, N., Boffi, D., Gastaldi, L.: Unfitted mixed finite element methods
  for elliptic interface problems. Numerical Methods for Partial Differential
  Equations  \textbf{40},  e23063 (2024). \doi{10.1002/num.23063}

\bibitem{auricchio_fictitious_2015}
Auricchio, F., Boffi, D., Gastaldi, L., Lefieux, A., Reali, A.: On a fictitious
  domain method with distributed {Lagrange} multiplier for interface problems.
  Applied Numerical Mathematics  \textbf{95},  36--50 (2015).
  \doi{10.1016/j.apnum.2014.05.013}

\bibitem{schilders_preconditioning_2008}
Benzi, M., Wathen, A.J.: Some preconditioning techniques for saddle point
  problems. In: Schilders, W.H.A., Van Der~Vorst, H.A., Rommes, J., Bock, H.G.,
  De~Hoog, F., Friedman, A., Gupta, A., Neunzert, H., Pulleyblank, W.R.,
  Rusten, T., Santosa, F., Tornberg, A.K., Bonilla, L.L., Mattheij, R.,
  Scherzer, O. (eds.) Model {Order} {Reduction}: {Theory}, {Research} {Aspects}
  and {Applications}, Mathematics in {Industry}, vol.~13, pp. 195--211.
  Springer Berlin Heidelberg, Berlin, Heidelberg (2008).
  \doi{10.1007/978-3-540-78841-6}

\bibitem{boffi_mixed_2013}
Boffi, D., Brezzi, F., Fortin, M.: Mixed finite element methods and
  applications. Springer Berlin Heidelberg, Berlin, Heidelberg (2013)

\bibitem{boffi_finite_2015}
Boffi, D., Cavallini, N., Gastaldi, L.: The finite element immersed boundary
  method with distributed lagrange multiplier. SIAM Journal on Numerical
  Analysis  \textbf{53},  2584--2604 (2015). \doi{10.1137/140978399}

\bibitem{conditioning}
Boffi, D., Credali, F., Gastaldi, L.: On stability and conditioning of a
  fictitious domain formulation for fluid-structure interaction problems
  (2025), in preparation

\bibitem{boffi_parallel_2024}
Boffi, D., Credali, F., Gastaldi, L., Scacchi, S.: A parallel solver for
  fluid–structure interaction problems with {Lagrange} multiplier.
  Mathematics and Computers in Simulation  \textbf{220},  406--424 (2024).
  \doi{10.1016/j.matcom.2024.01.027}

\bibitem{boffi_mixed_2014}
Boffi, D., Gastaldi, L., Ruggeri, M.: Mixed formulation for interface problems
  with distributed {Lagrange} multiplier. Computers \& Mathematics with
  Applications  \textbf{68},  2151--2166 (2014).
  \doi{10.1016/j.camwa.2014.07.020}

\bibitem{briggs_multigrid_2000}
Briggs, W.L., Henson, V.E., McCormick, S.F.: A multigrid tutorial. Society for
  Industrial and Applied Mathematics, Philadelphia, Pa (2000)

\bibitem{vanka_block-implicit_1986}
Vanka, S.: Block-implicit multigrid calculation of two-dimensional
  recirculating flows. Computer Methods in Applied Mechanics and Engineering
  \textbf{59},  29--48 (1986). \doi{10.1016/0045-7825(86)90022-8}

\end{thebibliography}

\end{document}